\documentclass[12pt,a4paper,english]{scrartcl}

\usepackage{lmodern}
\usepackage{babel,csquotes,xpatch}
\usepackage[
    backend=biber,
    style=alphabetic,
    backref=true
]{biblatex}
\usepackage{enumitem}
\setenumerate{label=\textup{(\roman*)}}

\usepackage{amsmath, mathtools}
\usepackage{amssymb}
\usepackage{amsthm, thmtools}
\usepackage{tikz-cd}
\usepackage{xcolor}
\usepackage{braket}
\usepackage{bm, bbm}
\usepackage{enumitem}
\usepackage{booktabs}
\usepackage{nicematrix}

\usepackage{microtype}

\usepackage{etoolbox}

\usepackage[colorlinks]{hyperref}
\usepackage[nameinlink]{cleveref}

\makeatletter
\renewcommand*\author[1]{%
  \stepcounter{author}%
  \ifnum\c@author=1
    \gdef\@author{#1}%
  \else
    \xdef\@author{\unexpanded\expandafter{\@author\and#1}}%
  \fi
  \csgdef{author@\the\c@author}{#1}}
\newcommand*\email[1]{%
  \csgdef{email@\the\c@author}{#1}}
\newcommand*\address[1]{%
  \csgdef{address@\the\c@author}{#1}}
\AtEndDocument{%
  \xdef\author@count{\the\c@author}%
  \c@author=1
  \print@authors}
\newcommand*\print@authors{%
  \ifnum\c@author>\author@count
  \else
    \print@author{\the\c@author}%
    \advance\c@author by 1
    \expandafter\print@authors
  \fi}
\newcommand*\print@author[1]{%
  \par\medskip
  \begin{tabular}{@{}l@{}}%
    \textsc{Addresses of \csuse{author@#1}}\\
    \csuse{address@#1}\\
    \textit{E-mail address}:
    \href{mailto:\csuse{email@#1}}{\texttt{\csuse{email@#1}}}
  \end{tabular}}
\makeatother

\declaretheorem[numberwithin=section]{theorem}
\declaretheorem[numberlike=theorem]{lemma, corollary,proposition}
\declaretheorem[numberlike=theorem,style=definition]{definition}
\declaretheorem[numberlike=theorem,style=remark,qed=$\diamondsuit$]{example}
\declaretheorem[numberlike=theorem,style=remark]{remark}

\declaretheoremstyle[notefont=\bfseries, notebraces={(}{)}]{conjecture}
\declaretheorem[numberlike=theorem,style=conjecture,refname={conjecture,conjectures},Refname={Conjecture,Conjectures}]{conjecture}

\newcommand{\NN}{\mathbb{N}}
\newcommand{\ZZ}{\mathbb{Z}}
\newcommand{\QQ}{\mathbb{Q}}
\newcommand{\RR}{\mathbb{R}}
\newcommand{\CC}{\mathbb{C}}
\newcommand{\KK}{\mathbb{K}}
\newcommand{\PP}{\mathbb{P}}

\newcommand{\VV}{\mathcal{V}}
\renewcommand{\AA}{\mathbb{A}}
\newcommand{\mm}{\mathfrak{m}}
\newcommand{\OO}{\mathcal{O}}

\DeclareMathOperator{\codim}{codim}

\DeclareMathOperator{\mult}{mult}
\DeclareMathOperator{\ord}{ord}
\newcommand{\reg}{\mathrm{reg}}

\DeclareMathOperator{\rank}{rk}

\DeclareMathOperator{\Matroid}{M}

\DeclareMathOperator{\Res}{Res}
\DeclareMathOperator{\Chow}{Ch}
\DeclareMathOperator{\PGr}{\mathbb{G}}
\DeclareMathOperator{\Ker}{Ker}

\DeclareMathOperator{\len}{len}
\DeclareMathOperator{\Row}{Row}
\DeclareMathOperator{\Spec}{Spec}

\DeclareMathOperator{\assgr}{gr}
\DeclareMathOperator{\TC}{TC}

\title{Splitting the Matroid Determinant}
\author{Clara Briand}
\address{École Normale Supérieure, France}
\email{clara.briand@ens.psl.eu}

\author{Leonie Kayser}
\address{Max Planck Institute for Mathematics in the Sciences, Leipzig\\
and University of Bern, Switzerland}
\email{leonie.kayser@unibe.ch}

\author{Julian Weigert}
\address{Max Planck Institute for Mathematics in the Sciences, Leipzig\\and Georg-August-Universität Göttingen, Germany}
\email{julian.weigert@mis.mpg.de}

\date{\today}

\begin{document}

\maketitle

\begin{abstract}
The principal matroid determinant $E_L$ of a linear space $L \subseteq\PP^n$ has been introduced in recent work by Matsubara-Heo and Telen. In this paper, we give the complete factorization of this polynomial into its irreducible components, proving a conjecture of the aforementioned authors. Our methods rely on bounding local multiplicities and an étale-local description of the strata of reciprocal linear spaces developed by Elias, Proudfoot and Wakefield. We also discuss analogous questions for other coordinate-wise powers of linear spaces.
\end{abstract}

\tableofcontents

\section{Introduction}
\label{sec:intro}
%%%%%%%%%%%%%%%%%%%%%%%%%%%%%%%%%%%%%%%%%%%%%%%%%%%cc

The theory of discriminants, resultants and determinants, carefully developed by Gelfand, Kapranov and Zelevinsky in \cite{GKZ}, plays a central role in algebraic geometry, particularly in applications. Loosely speaking, a discriminant is a polynomial that vanishes whenever a geometric object of study behaves non-generically. A particularly rich class of examples arises from toric geometry: given a finite collection of lattice points $A = \{a_0,\dots, a_n\} \in \ZZ^d$, the equation $\Delta_A$ of the projective dual $X_A^\vee$ of the associated toric variety vanishes on those points $(c_a)_{a\in A} \in (\PP^n)^*$ for which the hypersurface defined by the polynomial $f_c(x)\coloneq\sum_{a\in A}c_ax^a\in \CC[x_0^{\pm1},\dots,x_n^{\pm1}]$ is singular. 

A beautiful aspect of the toric case is the behavior of the associated \emph{$A$-discriminants} and \emph{$A$-determinants} under the natural stratification of toric varieties into torus orbits \cite[Chapter 10]{GKZ}. A similarly combinatorial stratification arises when studying \emph{reciprocal linear spaces} $L^{-1}$, the image closure of linear spaces under the Cremona transform. Such reciprocal linear spaces decompose similarly according to the flats of the associated matroid $M\coloneq\Matroid(L)$ \cite{Proudfoot2006}. Building on this idea, Matsubara-Heo and Telen propose the study of the \emph{principal matroid determinant} $E_L \subseteq (\PP^n)^*$ \cite[Definition 3.2] {MatsubaraHeoTelen2026}, the central object of study in this article. This hypersurface is a substitution of the Chow form of the projective variety $L^{-2}$. A hyperplane $H \in (\PP^n)^*$ belongs to the vanishing locus of $E_L$ whenever the intersection $(\CC^\times)^n\cap L^{-1}\cap H$ does not have the generic Euler characteristic \cite[Theorem 1.2.4]{MatsubaraHeoTelen2026}. This directly parallels similar result in the toric setup due to Esterov \cite[Theorem 1.8]{Esterov2013Discriminant}. The construction of the principal matroid determinant naturally generalizes to negative powers of linear spaces. We thus define \emph{generalized matroid determinants} $E_{L,-k}$ for any $k \geq 1$ (see \Cref{def:matroid_det} below for details), with $E_{L,-2} = E_L$ being the principal matroid determinant.

While many properties of principal $A$-determinants directly generalize to principal matroid determinants, one particularly interesting property remained open in \cite[Conjecture 7.1]{MatsubaraHeoTelen2026}: the factorization of the principal matroid determinant $\Delta(L^1_F)$ as a product of matroid discriminants with multiplicities. The aim of this article is to fill this gap by proving the following theorem.

\begin{theorem}
\label{thm:intro_main}
Let $L \subseteq \PP^n$ be a linear space with loopless matroid $M=\Matroid(L)$ and let $k\geq 2$. The generalized matroid determinant $E_{L,-k}$ decomposes into irreducible factors as
\[
E_{L,-k} = \prod_{F \textup{ conn.\ flat of }M} \Delta(L_{F}^{-k+1})^{m_F}, \qquad m_F = k^{\rank(M/F)-c(M)+1}\mu^+(M/F).
\]
The multiplicity $m_F$ further equals the local multiplicity of $L^{-k+1}$ along the stratum indexed by $F$. In particular, \cite[Conjecture 7.1]{MatsubaraHeoTelen2026} holds true.
\end{theorem}

While the components of the hypersurface $\VV(E_{L,-k})$ were already known set-theoretically for $k=2$\cite[Theorem 1.2]{MatsubaraHeoTelen2026}, our work contributes a precise description of the exponents $m_F$, both in terms of matroid invariants and in terms of the local geometry of $L^{-k}$ for any $k\geq 2$.
Our main proof can be outlined as follows.
\begin{enumerate}[label=(\arabic*)]
\item In \Cref{sec:lower bound} we show that $m_F$ is a lower bound for the multiplicity of the factor $\Delta(L_F^{-k+1})$ inside $E_{L,-k}$. To obtain this lower bound we use the definition of $E_{L,-k}$ as a pullback of the Chow form of $L^{-k}$. We first track how local multiplicities may change when passing from $L^{-k}$ to the Chow incidence variety and then through the projection to the Chow form and it's pullback to $E_{L,-k}$. In this way we find that the local multiplicity of $L^{-k}$ along the stratum indexed by a flat $F$ is a lower bound to the multiplicity of the factor in $E_{L,-k}$.
\item Next, in \Cref{sec:local mult of L-k} we compute the local multiplicity of $L^{-k}$ along every flat stratum exactly: We first translate this into a local computation on $L^{-k+1}$ and then use that locally around a point in a flat stratum, $L^{-k+1}$ splits into a product of smaller reciprocal linear spaces via an étale map due to \cite{ELIAS201636}.
\item Finally, in \Cref{sec:match} we argue that the multiplicities predicted by  \Cref{thm:intro_main} sum up to the correct degree and hence the lower bounds must all be equalities. For the last step we use Möbius inversion, as well as the formulas for the degrees of $\Delta(L_F^{-k+1})$ proven in our previous work \cite[Corollary 1.4]{briand2025}.
\end{enumerate}

This manuscript is organized around the proof of the main result in \Cref{sec:main proof}. The preliminary \Cref{sec:prelim} recalls and generalizes the definition of principal matroid determinants, comparing it to the toric story, as well as providing a brief primer on local multiplicities in commutative algebra. After the proof of the main theorem, in \Cref{sec:otherMulties} we make first steps towards studying the analogously constructed generalized matroid determinants of $L^k$, $k \geq -1$ and state analogous conjectures about their factorization.

\subsection*{Acknowledgements}
C.B.\ thanks the Max Planck Institute for Mathematics in the Sciences for its hospitality and the Fondation de l'ENS for its financial support. J.W.\ was supported by the SPP 2458 “Combinatorial Synergies”, funded by the DFG grant 539677510.

\section{Preliminaries}\label{sec:prelim}

Throughout this article, we work over an algebraically closed field $\KK$ of characteristic zero. All rings considered are Noetherian and there is no harm in assuming them to be essentially of finite type over $\KK$.

\subsection{Generalized matroid determinants}

In this section, we give a brief introduction to discriminants from the point of view of projective duality, an approach pioneered in \cite{GKZ}. We then recall the construction of the principal $A$-determinant $E_A$ as well as its matroidal analogue, namely, the principal matroid determinant $E_L$ introduced in \cite{MatsubaraHeoTelen2026} and its generalizations.

\begin{definition}
Let $X \subset \PP^n$ be an irreducible projective variety. Its \emph{dual variety} is the irreducible variety
\[
    X^\vee \coloneq \overline{\Set{H \in (\PP^n)^* | \exists x \in X_{\text{reg}} \text{ such that } H \supseteq T_xX}}.
\]
If $\codim(X^\vee) = 1$, then there exists a unique, up to scaling, irreducible polynomial $\Delta_X$ such that $X^\vee = \VV(\Delta_X)$ called the \emph{$X$-discriminant}. If $\codim X^\vee >1$ then we set $\Delta_X \coloneq 1$.
\end{definition}

$X$-discriminants will be the building blocks of the principal determinants.
Another important construction is that of the $X$-resultant. Recall that for any $0 \leq k \leq n$, $\PGr(k,\PP^n)$ parametrizes projective linear spaces of dimension $k$ in $\PP^n$. 
Suppose $\dim X = d$. Then, generically, a linear space $[\Lambda] \in \PGr(n-d-1,\PP^n)$, of codimension $d+1$, will not intersect $X$. The \emph{Chow variety} of $X$ is formed by those linear spaces which do:
\[
    \mathcal{C}(X) \coloneq \Set{ [\Lambda] \in \PGr(n-d-1,\PP^n) | \Lambda \cap X \neq \emptyset}.
\]
The Chow variety is an irreducible hypersurface in $\PGr(n-d-1,\PP^n)$ \cite[Proposition~2.2]{GKZ}. It can be defined by a single polynomial $\Chow(X)$ in the Plücker variables, uniquely up to the Plücker relations. This polynomial of the same degree as $X$ is the \emph{Chow form} of $X$. 
There is another natural system of coordinates on the Grassmannian. Namely, consider the rational map
\[
r\colon \KK^{(d+1)\times(n+1)} \dashrightarrow \PGr(n-d-1,\PP^n), \qquad U \mapsto [\PP(\Ker U)], 
\]
which is defined the dense open subset of full-rank matrices, the (non-compact) Stiefel manifold. The entries of $U$ are called the \emph{Stiefel coordinates} on the Grassmannian; the map $r$ essentially amounts to taking all $(d+1)$-minors of $U$. Writing the Chow form $\Chow(X)$ in these coordinates yields the $X$-resultant.
\begin{definition}
The \emph{$X$-resultant} is $\Res(X) \coloneq r^*\Chow(X) \in \KK[u_{ij} \mid 0 \leq i \leq d, 0 \leq j \leq n]$.
\end{definition}

The case where $X$ is a toric variety has been extensively studied in \cite{GKZ}. We give here a short overview of some of the main constructions and results in this setting.

\begin{definition}
Let $A \in \ZZ^{(d+1)\times(n+1)} $ be of rank $d+1$, such that $(1,\dots,1) \in \Row(A)$. Let $a_0,\dots,a_n \in \KK^{d+1}$ be the columns of $A$. The projective toric variety $X_A$ associated to $A$ is the closure of the image of the map
\[
  (\KK^\times)^{d+1} \to  \PP^n , \qquad x \mapsto  [x^{a_0}:\dots:x^{a_n}]
\]
where $x^{a_i} \coloneq x_0^{a_{i,0}}\dotsm x_d^{a_{i,d}}$ are the Laurent monomials with exponent vectors $a_i$.
\end{definition}

The $X_A$-discriminant and the $X_A$-resultant are known as the \emph{$A$-discriminant} and \emph{$A$-resultant}, respectively. The \emph{principal $A$-determinant} is then defined as a specialization of the $A$-resultant.
\begin{definition}
For $z \in \KK^{n+1}$ consider the matrix $A_z \coloneq A\cdot\operatorname{diag}(z) \in \KK^{(d+1)\times(n+1)}$ and let $i\colon \KK^{n+1} \to \KK^{(d+1)\times(n+1)}$ be the map linear map $z \mapsto A_z$. The \emph{principal $A$-determinant} is 
\[
    E_A \coloneq i^*\Res(X_A) = i^*r^*\Chow(X_A) \in \KK[z_0,\dots,z_n].
\]
\end{definition}

For $\KK=\CC$, Esterov showed that the zero locus of $E_A$ consists of all the points $z \in (\PP^n)^*$ such that the topological Euler characteristic of the hyperplane section $|\chi(X_A \cap \PP((\CC^\times)^{n+1})\cap H_z)|$ is less than its generic value \cite[Theorem 1.8]{Esterov2013Discriminant}, see also \cite[Section 3]{Fevola2024}. Here, $H_z$ is the hyperplane $\set{x | z_0x_0+\dots+z_nx_n=0} \subseteq \PP^n$.

The prime factorization of $E_A$ can be described in terms of the lattice polytope $P_A \coloneq \text{Conv}(a_0,\dots,a_n)\subset\RR^{d+1}$ associated to $X_A$. Recall that a toric variety admits a stratification by torus orbits. These strata are identified with faces of $P_A$ via the orbit-cone correspondence, and they give rise to the factors of $E_A$. Let $\mathrm{Vol}(-)$ denotes the (normalized) lattice volume.

\begin{theorem}[{\cite[Theorem 10.1.2]{GKZ}}]\label{th: ppalAdet}
The principal A-determinant has degree 
$
\deg(E_A) = (d+1)\mathrm{Vol}(P_A).
$
It factors as
\[
E_A = \prod_{Q \leq P_A} \Delta(X_{A|_Q})^{m_Q}
\]
where $Q$ ranges over the faces of $P_A$, and $m_Q = [P_A:Q]\mathrm{Vol}(Q)$.
\end{theorem}

For other varieties with a similar stratification, one may hope to construct a polynomial with similar properties to the principal $A$-determinant. This idea was pursued by  Matsubara-Heo and Telen in the construction of the \emph{principal matroid determinant} in \cite{MatsubaraHeoTelen2026}, which replaces toric varieties with reciprocal linear spaces.

\begin{definition}
Let $k \in \ZZ \setminus \{0\}$, and let $L \subset \PP^n$ be a linear space not contained in any coordinate hyperplane. The \emph{$k$-th power map} is 
\[
\phi_k \colon L \dashrightarrow \PP^n, \qquad x \mapsto x^k \coloneq [x_0^k:\dots:x_n^k].
\]
Define $L^k \coloneq \overline{\phi_k(L)}$: we say that $L^{-1}$ is a \emph{reciprocal linear space}.
\end{definition}

Notice that $\phi_k$ is well-defined everywhere for $k \geq 1$, but only rational for negative values of $k$.
We will say that  the $L^{-k}$-discriminant $\Delta(L^{-k})$ is the \emph{k-th generalized matroid discriminant} of $L$. For $k = 1$, this is the matroid discriminant defined in \cite{MatsubaraHeoTelen2026}.

From now on, we will fix $L \subset \PP^n$ to be a linear space of dimension $d$, not contained in any coordinate hyperplane, and a full-rank matrix $A \in \KK^{(d+1)\times(n+1)}$ such that $L = \PP(\Row A)$. Let $a_0,\dots,a_n$ be the columns of $A$.
Recall that the matroid $M$ associated to $L$ is the matroid on the ground set $E \coloneq \{0,\dots,n\}$ whose independent sets index linearly independent columns of $A$. The matroid $M$ does not depend on the choice of the representative matrix $A$. We will freely use standard terminology for matroids and refer to \cite{oxley2006matroid} for the precise definitions. We write $\mathcal{F}(M)$ for the set of all flats of $M$ and $c(M)$ for the number of connected components of $M$. Furthermore $\mu(M)$, $\mu^+(M)\coloneq |\mu(M)|$ and $\beta(M)$ denote the signed and unsigned Möbius invariant and the beta invariant of $M$ respectively.

In \cite[Proposition 5]{Proudfoot2006}, it was shown that reciprocal linear spaces admit a stratification into pieces isomorphic to (open parts of) smaller reciprocal linear spaces, which can be described in terms of the matroid $M$ associated to $L$. Namely, for $F \subset E$, let
\[
T_F \coloneq \set{x \in \PP^n | x_i = 0 \text{ iff } i \notin F}, \qquad \Lambda_F \coloneq \overline{T_F} = \set{x \in \PP^n | x_i = 0 \text{ if } i \notin F}.
\]
Then, Proudfoot and Speyer showed
\[
L^{-1} = \bigsqcup_{F \in \mathcal{F}(L)} L^{-1} \cap T_F \qquad  \text{and} \qquad L^{-1} \cap \Lambda_F = \overline{L^{-1} \cap T_F} \cong L_F^{-1},
\]
where $L_F$ is the projection of $L$ onto the coordinate subspace $\KK^F \subseteq \KK^E$ indexed by $F$. In fact, as shown in \cite[Proposition 3.10]{MatsubaraHeoTelen2026}, this decomposition holds for arbitrary negative powers of linear spaces, $k \geq 1$:
\[
L^{-k} = \bigsqcup_{F \in \mathcal{F}(L)} L^{-k} \cap T_F \qquad  \text{and} \qquad L^{-k} \cap \Lambda_F = \overline{L^{-k} \cap T_F} \cong L_F^{-k}.
\]
By slight abuse of notation, we will thus write $L^{-k}_F=\overline{L^{-k}\cap T_F}$. 
We can mirror the whole construction of the principal $A$-determinant in this setting. Recall that $A$ is a matrix defining

\begin{definition}\label{def:matroid_det}
Consider again the map $i\colon \KK^{n+1} \to \KK^{(d+1)\times(n+1)}$, $z \mapsto  A\cdot\text{diag}(z)$. For $k \geq 2$, the \emph{$k$-th generalized matroid determinant} is 
\[
    E_{L,-k} \coloneq i^*\Res(L^{-k}) = i^*r^*\Chow(L^{-k}).
\]
\end{definition}
\begin{remark}
    The behavior of $i^*r^*\Chow(L^{k})$ for $k = -1$ or $k \geq 1$ is very different to that of the principal matroid determinant $E_{L,-k}$, so we decided \emph{not} to include those cases in the definition of the generalized matroid determinant, although the definition still makes sense. We include some results and open questions for this range of $k$ in \Cref{sec:otherMulties}.%: \Cref{lem: k=-1} describes the zero locus of $i^*r^*\Chow(L^{k})$ for $k=-1$, and \Cref{conj: positivek} predicts the zero locus of $i^*r^*\Chow(L^{k})$ for $k \geq 1$.
\end{remark}
Since $E_{L,-k}(z)$ is clearly homogeneous in $z$, it makes sense to consider its zero locus in projective space.
For $k=2$, we recover the principal matroid determinant $E_L$ introduced in \cite{MatsubaraHeoTelen2026}.  For $\KK=\CC$, Matsubara-Heo and Telen show that the zero locus of $E_L$ consists of all the points $z \in (\PP^n)^*$ such that the signed topological Euler characteristic of the hyperplane section $(L^{-1} \cap T_E)\cap H_z$ is less than its generic value \cite[Theorem~6.1]{MatsubaraHeoTelen2026}.

As shown in this article, the prime factorization of $E_{L,-k}$ is similar to the one of $E_A$ described in \Cref{th: ppalAdet}. The role of $P_A$ and its faces is now played by the matroid $M = \Matroid(L)$ and its flats. The following result describes this factorization up to determining the multiplicities. 

\begin{theorem}\label{th: ppalmatdet}
Let $k \geq 2$.
The $k$-th generalized matroid determinant has degree 
$
\deg(E_{L,-k}) =
%\begin{cases}
    (d+1)k^{d-c(M)+1}\mu^+(M).
    %& \text{ if } k \leq -1 \\
%    (d+1)2^{d-c(M)+1} & \text{ if } k \geq 1.
%\end{cases}
$
It factors into irreducible factors as
\[
E_{L,-k} =
%\begin{cases}
    \prod_{F\text{ conn.\ flat of }M} \Delta(L_F^{-k+1})^{m_F}%& \text{ if } k \leq -2 \\
    %\Delta(L^{k+1}) & \text{ if } k \geq 1.
%\end{cases}
\]
for some integers $m_F \geq 1$.
\end{theorem}

For $k=1$, this is the statement of \cite[Theorem 1.2(1)+(2)]{MatsubaraHeoTelen2026} (minus the restriction to connected flats). Our proof is very similar, however, it is complicated by the difficulties that $\phi_k$ is not an isomorphism on the torus and by the fact that $L^{-k}$ can have singularities in the torus $T_E$.

\begin{proof}
The claim about the degree follows by exactly the same argument as \cite[Theorem 3.5]{MatsubaraHeoTelen2026}, simply replacing their second powers by $k$-th powers. As mentioned, the factorization follows closely the proof of \cite[Theorem 3.11]{MatsubaraHeoTelen2026}.
Consider the incidence variety
\begin{align*}
Y \coloneq \Set{(x,z)\in \PP^n\times (\PP^n)^* | A_z\cdot x=0\text{ and }x\in L^{-k}}.
\end{align*}
The proof of \cite[Lemma 3.7]{MatsubaraHeoTelen2026} does not require smoothness on the torus and therefore by exactly the same argument we obtain that the zero locus $E_{L,-k}$ is the image of $Y$ under the second projection $\pi_2\colon \PP^n\times(\PP^n)^*\to (\PP^n)^*$.
On the other hand $Y$ can be stratified according to the decomposition of $L^{-k}$ into flat strata:
\begin{align}
\label{eq: Ystrat}
Y=\bigsqcup_{F\in \mathcal{F}(M)}\underbrace{Y\cap(T_F\times (\PP^n)^*)}_{\eqcolon Y_F }
\end{align}
We proceed in three steps to complete the proof:
\begin{enumerate}
\item \label{1} For any flat $F$ we have $(L_F^{-k+1})^\vee\subseteq \overline{\pi_2(Y_F)}$.
\item \label{2} If $F$ is a connected flat of $M$ then equality holds in \ref{1} and $\overline{\pi_2(Y_F)}$ is a hypersurface.
\item \label{3} If $F$ is a disconnected flat and $F'$ is a component of $F$ then $\overline{\pi_2(Y_F)}\subseteq \overline{\pi_2(Y_F')}$.
\end{enumerate}
Once these three claims are proven the result follows: Since by the above $\pi_2(Y)$ is the vanishing locus of $E_{L,-k}$ while by \eqref{eq: Ystrat} we also have
\begin{align*}
\mathcal{V}(E_{L,-k})=\pi_2(Y)=\bigcup_{F\in \mathcal{F}(M)}\overline{\pi_2(Y_F)}\overset{\ref{3}}{=} \bigcup_{F\text{ conn.\ flat of }M}\overline{\pi_2(Y_F)}\overset{\ref{2}}{=}\bigcup_{F\text{ conn.\ flat of }M}(L_F^{-k+1})^\vee.
\end{align*}
This yields all the irreducible factors of $E_{L,-k}$, as the dual hypersurfaces $(L_F^{-k+1})^\vee = V(\Delta(L_F^{-{k+1}}))$ on the right hand side are irreducible. It remains to prove  \ref{1}-\ref{3}.
\begin{enumerate}[wide]
\item Write $\ell_0,\ldots,\ell_n$ for the linear forms on $\KK^{d+1}$ whose coefficients in the standard basis are the columns of the matrix $A$, that is, $\ell(x) = x\cdot A$. Then $L_F^{-k+1}$ is rationally parametrized by the map $\ell^{-k+1}_F \colon \PP^d \dashrightarrow \PP^n$ whose $i$-th coordinate is $\ell_i^{-k+1}$ if $i\in F$ and $0$ otherwise. Similarly, $L_F^{-k}$ is rationally parametrized by $\ell^{-k}_F$. Assume $t\in \PP^d$ is such that $\ell^{-k+1}_F$ is defined at $t$ and $\ell^{-k+1}_F(t)$ is a smooth point of $L_F^{-k+1}$. Notice that in particular $\phi_{-k+1}$ is unramified at $t$, so we can use the parametrization $\ell^{-k+1}_F$ to compute the tangent space to $L_F^{-k+1}$ at $\ell^{-k+1}_F(t)$. The proof of \cite[Lemma 3.6]{MatsubaraHeoTelen2026} then shows that the pair $(\ell^{-k}_F(t),z)$ belongs to $Y_F$ if and only if the hyperplane defined by $z$ contains the tangent space to $L^{-k+1}_F$ at $\ell^{-k+1}_F(t)$. This implies that $(L_F^{-k+1})^\vee\subseteq \overline{\pi_2(Y_F)}$. 
\item Notice that $A_z\cdot x=A_x\cdot z$ and hence the fiber in $Y$ over a fixed point $x\in L^{-k}$ is a linear space $\Lambda_x$. The dimension of $\Lambda_x$ is $n-\rank(A_x)$ and hence it only depends on which coordinates of $x$ vanish, since a non-zero rescaling of columns of $A$ does not alter the rank. More precisely we have $\rank(A_x)=\rank_M(F)$ whenever $x\in T_F$. In particular, $Y_F$ the total space is a vector bundle over $L^{-k}\cap T_F$ and therefore irreducible of dimension 
\[
\dim Y_F = \dim L_F^{-k+1}+\dim \Lambda_x =\rank_M(F)-1+n-\rank_M(F)=n-1.
\]
In particular $\overline{\pi_2(Y_F)}$ is an irreducible variety of dimension at most $n-1$. On the other hand, assuming $F$ is connected, we have that $(L_F^{-k+1})^\vee$ is a hypersurface by \cite[Theorem 4.8]{briand2025}. Hence equality must hold in \ref{1} in that case.
\item Suppose $F$ is disconnected and write $F=F'\sqcup G$ where $F'$ is a connected flat of $M$. By an invertible change of coordinates in $\operatorname{GL}_{d+1}(\KK)$ we can assume that the columns of $A$ whose index belongs to $F$ are of block shape, i.e.\ $A$ is of the form displayed below where the first columns belong to $F'$, the second set of columns belongs to $G$ and the last set is in $F^c$. This does not change $Y$.
\[
    A=
\begin{bNiceArray}{rrr|rrr|rrr}[margin=4pt]
\Block{2-3}{A_1} & & & \Block{2-3}{0} & & & \Block{2-3}{A_3} & & \\
& & & & & & & & \\
\hline
\Block{2-3}{0} & & & \Block{2-3}{A_2} & & & \Block{2-3}{A_4} & & \\
& & & & & & & & \\
\CodeAfter
  \UnderBrace{4-1}{4-3}{F'}[shorten]
  \UnderBrace{4-4}{4-6}{G}[shorten]
   \UnderBrace{4-7}{4-9}{F^c}[shorten]
\end{bNiceArray}\vspace{3ex}
\]
Now suppose $z\in \pi_2(Y_F)$, then there exists $x\in T_F\cap L^{-k+1}$ such that $A_z\cdot x=0$. But then setting the entries $x_i$ with $i\in G$ to zero we obtain a point $x'\in L^{-k+1}\cap T_{F'}$ which still satisfies $A_z\cdot x'=0$ due to the block shape of $A$. Hence $z\in \pi_2(Y_{F'})$ which finishes the proof. \qedhere
\end{enumerate}
\end{proof}

The description of $E_{L,-k}$ in \Cref{th: ppalmatdet} is incomplete without a closed formula for the exponents $m_F$. Our main result \Cref{thm:intro_main} provides such a formula; \Cref{sec:main proof} is devoted to the proof of this theorem.

\subsection{Local multiplicity}

We recall some properties of local multiplicity needed for the proof of \Cref{thm:intro_main}. For reference, see for example \cite[Chapter VII.10]{ZariskiSamuelVol2} or \cite[Chapter 11]{Swanson2006}. All rings appearing in this section are assumed to be Noetherian.

Let $(R,\mm_R)$ be a Noetherian local ring and $I$ an $\mm_R$-primary ideal. Under this assumption, the quotients $I^i/I^{i+1}$ are $R$-modules of finite length. The Hilbert-Samuel function of $(R,I)$ is the Hilbert function of the associated graded ring $\operatorname{gr}_I(R) \coloneq \bigoplus_{i\geq 0} I^i/I^{i+1}$
\[
\chi_R^I(i) \coloneq \operatorname{hf}_{\operatorname{gr}_I(R)}(i) \coloneq \len_R(I^i/I^{i+1}).
\]
This function agrees with the Hilbert polynomial $\operatorname{HP}_{I,R}(t) \in \QQ[t]$ for $i\gg 0$. If $\dim R = n$, then $\deg \operatorname{HP}_{I,R} = n-1$.

\begin{definition}[Local multiplicity]
The multiplicity of the $\mm_R$-primary ideal $I \subset R$, denoted by $e(I,R) \in \NN_{>0}$, is the normalized leading coefficient of the Hilbert polynomial
\[
\operatorname{HP}_{I,R}(t) = \frac{e(I,R)}{(n-1)!}t^{n-1} + O(t^{n-2}).
\]
Equivalently, writing the Hilbert series $\sum_{i\geq 0} \chi_R^I(i)t^i$ as $\frac{Q(t)}{(1-t)^n}$, then $e(I,R) = Q(1)$.

For any Noetherian ring $R$, $I\subseteq R$ a $\mathfrak{p}$-primary ideal, we set $e(I,R)\coloneq e(IR_\mathfrak{p},R_\mathfrak{p})$.

If $X$ is a variety, or more generally an algebraic $\KK$-scheme, and $Z \subseteq X$ a closed irreducible subset, then the \emph{local multiplicity} of $X$ along $Z$ is $\mult_Z(X) \coloneq e(\mm_{X,Z},\OO_{X,Z})$ where $\mm_{X,Z}$ denotes the maximal ideal of the local ring $\OO_{X,Z}$ associated to $Z$.
\end{definition}

The local multiplicity satisfies the following properties, which we will use later on.

\begin{proposition}
\label{prop: locmult}
\begin{enumerate}
\item \label{prop: locmult1}  Let $J \subset I \subset R$ be $\mm_R$-primary ideals, then $e(I,R) \leq e(J,R)$.
\item \label{prop: locmult2} Let $f: R \rightarrow R'$ be an étale morphism of local rings with the same residue field, and let $I \subset R$ be an $\mm_R$-primary ideal. Denote $I' \coloneq f(I)R' \subset R'$. Then $e(I,R) = e(I',R')$.
\item \label{prop: locmult3} Let $R_1,R_2$ be $\KK$-algebras, $\mm_i \subseteq R_i$ maximal ideals with $R_i/\mm_i=\KK$, and let $I_i \subseteq R_i$ be $\mm_i$-primary. Let $I=I_1\otimes R_2+R_1\otimes I_2 \subseteq R \coloneq R_1 \otimes R_2$, then $I$ is ($\mm_1\otimes R_2+R_1\otimes \mm_2$)-primary and $e(I,R) = e(I_1,R_1)e(I_2,R_2)$.
\end{enumerate}
\end{proposition}

\begin{proof}
\begin{enumerate}[wide]
\item Assume $e(I,R) > e(J,R)$, then $\chi_R^I(t) = \chi_R^J(t) + \frac{e(I,R) - e(J,R)}{(n-1)!}t^{n-1} + O(t^{n-2})$, contradicting the estimate
\[
\sum_{j=0}^i \chi_R^I(j) = \len_R(R/I^i) \leq \len_R(R/J^i) =\sum_{j=0}^i \chi_R^J(j).
\]
% \item For any $k \geq 0$, we have the isomorphism of $R$-modules
% \[
% I^k/I^{k+1} \cong (I^{k}/I^{k+2})/(I^{k+1}/I^{k+2})
% \]
% and therefore
%  \[
% \len_R I^k/I^{k+1} + \len_R I^{k+1}/I^{k+2} = \len_R I^{k}/I^{k+2}
% \]
% By using this identity recursively we obtain that for any $k'>k$
% \[
% \len_R (I^{k}/I^{k+k'}) = \len_R (I^k/I^{k+1}) + \cdots + \len_R (I^{k+k'-1}/I^{k+k'})
% \]
% For $k=0$ this yields
% \[
% \len_R (R/I^{k'}) = \len_R (R/I^{1}) + \cdots + \len_R (I^{k'-1}/I^{k'})
% \]
% Since $J \subset I$, $\len_R (R/I^{k'}) \leq \len_R (R/J^{k'})$ and so
% \[
% \len_R (R/I^{1}) + \cdots + \len_R (I^{k'-1}/I^{k'}) \leq \len_R (R/J^{1}) + \cdots + \len_R (J^{k'-1}/J^{k'})
% \]
% This holds for any $k' > 0$. Now,
% by definition of the Hilbert polynomial, we see that asymptotically, the left hand side of this equation grows like 
% \[
% C + \frac{e(I,R)}{(n-1)!}(k_0^{n-1}+\cdots+(k')^{n-1}) \sim
% \frac{e(I,R)}{(n-1)!}\frac{(k')^{n}}{n}
% \]
% for some constant $C > 0$ and some $k_0 \in \mathbb{N}$,
% and similarly the right hand grows asymptotically like 
% \[
% \frac{e(J,R)}{(n-1)!}\frac{(k')^{n}}{n}
% \]
% We conclude that $e(J,R) \geq e(I,R)$.
\item Since $R$ and $R'$ have the same residue fields, $f$ being étale implies that $f$ induces an isomorphism of the completions of $R$ and $R'$, see \cite[Proposition 4.3.26]{Liu2002}. The statement follows from the fact that Hilbert-Samuel multiplicity is preserved under completion \cite[Chapter VII, Lemma 8.1]{ZariskiSamuelVol2}.
\item With the given assumptions and notation, we have an isomorphism of associated graded algebras
\begin{align*}
    \assgr_I(R)=\assgr_{I_1\otimes R_2+R_1\otimes I_2}(R\otimes R')=\assgr_{I_1}(R)\otimes \assgr_{I_2}(R').
\end{align*}
Let $\mm = \mm_1\otimes R_2+R_1\otimes \mm_2$, then the Hilbert series of $I,I_1,I_2$ satisfy
\[
\operatorname{HS}_{I_\mm,R_\mm}=\operatorname{HS}_{(I_1)_{\mm_1},(R_1)_{\mm_1}} \cdot \operatorname{HS}_{(I_2)_{\mm_2},(R_2)_{\mm_2}}.
\]
From the description $e(I,R) = ((1-t)^{\dim R_\mm} \operatorname{HS}_{I_\mm,R_{\mm}}(t))(1)$, we see that the multiplicities multiply. \qedhere
\end{enumerate}
\end{proof}

The associated graded ring of the local ring $\mathcal{O}_{X,x}$ of a point of a scheme is the coordinate ring of the tangent cone $\operatorname{TC}_xX = \Spec \assgr_{\mm_{X,x}}(\OO_{X,x})$. In particular, the Hilbert function and -polynomial of $\mathcal{O}_{X,x}$ are that of the projectivized tangent cone $\PP(\operatorname{TC}_xX) \hookrightarrow \PP(\operatorname{T}_xX)$, so by definition $\mult_x X = \deg \PP(\operatorname{TC}_xX)$.

\begin{example}\label{ex: multinRLS}
Let $X = \widehat{L^{-1}} \subseteq \AA^E$ and let $x \in X$ be in the stratum $X \cap T_F$ corresponding to the flat $F$ (that is, $x_i \neq 0$ if and only if $i \in F$). In \cite[Theorem 24]{Sanyal2013}, the authors show that the tangent cone $\TC_0(X-x) \subseteq \TC_0 \AA^E = \AA^E$ has the explicit description
\[
\TC_0(X-x) = \widehat{L|_F} \times \widehat{(L/F)^{-1}} \subseteq \AA^F \times \AA^{E\setminus F}.
\]
In particular, since $\widehat{L|_F}$ is a linear space, $\PP(\TC_xX)$ is a cone over $(L/F)^{-1}$, hence
\[
\mult_x X = \deg \PP(\TC_xX) = \deg (L/F)^{-1} = \mu^+(M/F).\qedhere
\]
\end{example}

Geometrically the local multiplicity enjoys the following properties, which we will use throughout the article:

\begin{proposition}\label{prop: propertiesmult}
Let $X$ be an algebraic scheme, $Z \subseteq X$ an irreducible closed subset.
\begin{enumerate}
\item \label{prop: propertiesmult1} If $X$ is reduced and equidim.\ at $x\in X$, then $\mult_x X = 1$ if and only if $x \in X_\reg$.
\item \label{prop: propertiesmult2} If $X$ is equidimensional, then the map $Y\mapsto\mult_{Y}(X)$ is upper semi-continuous on $X$ (including its non-closed points). In particular,
\[
\mult_Z(X) = \min_{p \in Z} \mult_p X = \mult_x X \text{ for $x$ in a dense open $U \subseteq Z$}.
\]
\item \label{prop: propertiesmult3} If $Y$ is another scheme, $x \in X$ and $y\in Y$, then $\mult_{(x,y)} X\times Y = \mult_xX\cdot \mult_y Y$.
\item \label{prop: propertiesmult4} If $X \subseteq \PP^n$ is (quasi-)projective, then $\mult_Z X = \mult_{\widehat Z} \widehat{X}$.
\item \label{prop: propertiesmult5} Let $W \subseteq X$ be a subscheme locally around $Y$ defined by $f \in \OO_{X,Y}$, that is, $\mathcal{I}(W)\OO_{X,Y} = f\mathcal{O}_{X,Y}$. Assume that $\OO_{X,Y}$ is regular, then
\[
\mult_Y(W) = \ord_{\mm_{X,Y}}(f) \coloneq \max \set{ k \in \NN | f \in \mm_{X,Y}^k}.
\]
\end{enumerate}
\end{proposition}

\begin{proof}
\begin{enumerate}[wide]
\item This is evident from the previous tangent cone description. The general statement is due to Nagata \cite[Theorem 40.6]{Nagata1962}: A Noetherian local ring is regular if and only if $e(R)=1$ and $A$ is quasi-unmixed ($\dim \widehat{R}/\mathfrak{p} = \dim \widehat{R}$ for all associated primes of $0_{\widehat{R}}$). The latter condition is always satisfied for local rings of varieties, so $\mult_x(X) = 1$ if and only if $\OO_{X,x}$ is regular.
\item The semi-continuity of local multiplicity is attributed to Dade \cite{dade1960multiplicity}. A precise statement can be found in \cite[Theorem 6.12]{VillamayorU2014}, where the result is discussed in the generality of excellent schemes.
\item This is a geometric reformulation of \Cref{prop: locmult}\ref{prop: locmult3}.
\item Locally, $\widehat{X}$ is isomorphic to an open set of $X \times \AA^1$, more precisely, $\widehat{X}$ is the open subset of the total space of $\mathcal{O}_X(-1)$ given by the complement of the zero section. Hence we can (locally) apply (iii), using that $\AA^1$ is regular.
\item This is \cite[Example 11.2.8]{Swanson2006}.
\qedhere
\end{enumerate}
\end{proof}

\begin{example}
\label{ex:multiesgeovsalg}
Let $W = V(f) \subseteq \AA^n$ be a hypersurface, and consider the irredundant factorization $f = f_1^{e_1}\dotsm f_m^{e_m}$ with $f_i$ irreducible. Let $Y = V(f_i)$, then $f\mathcal{O}_{\AA^n,Y} = f_i^{e_i} \OO_{\AA^n,Y} = \mm_{\AA^n,Y}^{e_i}$, so $\mult_Y(W) = e_i$. In this way, the local multiplicity recovers the standard notion of algebraic multiplicity of a factor in a polynomial.
\end{example}

We end this section with a theorem which will later be useful.

\begin{theorem}[{\cite[Ch. VIII, Section 10, Theorem 24, Corollary 1]{ZariskiSamuelVol2}}] \label{thm: ZariskiSamuel}
Let $R$ be a local ring, with maximal ideal $\mm_R$ and consider an inclusion of rings $R \hookrightarrow S$ which makes the semi-local ring $S$ a finite $R$-module. Let $J$ be an ideal of $R$ which is $\mm_R$-primary. Let $\mathfrak{m}_{S,1},\dots,\mathfrak{m}_{S,r}$ be the maximal ideals of $S$ and let $I_i$ be the $\mathfrak{m}_{S,i}$-primary component of $JS$. Suppose no nonzero element of $R$ is a zero divisor in $S$ and for all $i=1,\ldots,s$, the local ring $S_{\mm_{S,i}}$ has the same dimension as $R$. 
Then,
\[
[S:R]e(J, R) = \sum_{i=1}^r [S/\mm_{S,i}:R/\mm_R] e(I_i, S).
\]
\end{theorem}

\section{Proof of the main theorem}\label{sec:main proof}

In this section we give the proof of \Cref{thm:intro_main}. We start by illustrating this Theorem in two examples.

\begin{remark}\label{rem: simplifications}
By \cite[Theorem 4.8]{briand2025}, for $k \geq 2$, $\Delta(L_F^{-k+1})$ is dual-defective if and only if $F$ is not a connected flat, in which case it does not contribute to the factorization of $E_{L,-k}$. Therefore, the product in the factorization of $E_{L,-k}$ in \Cref{thm:intro_main} may be taken either  over connected flats only or over all flats where the disconnected flats simply contribute a factor of 1.
In particular, the component $\Delta(L^{-1})$ appears non-trivially in the factorization if and only if $\Matroid(L)$ is connected. It corresponds to the flat $F = E$, and so it appears with multiplicity
\[
m_E = k^{\rank(M/E)}\mu^+(M/E) = 1
\]
as $M/E$ is the empty matroid.
\end{remark}

\begin{example}
    Let $L = \PP(\operatorname{Row} A) \subset \PP^4$, where
    \[
    A = \begin{pmatrix}
        1 & 0 & -1 & 0 & 0 \\
        -1 & 1 & 0 & 0 & 1 \\
        0 & 0 & 0 & -1 & -1
    \end{pmatrix}.
    \]
    This matrix has rank $d+1$, its matroid $M=M(L)$ is the graphic matroid associated with the following graph.
    \begin{center}
\begin{tikzpicture}[scale=1.5]
    \coordinate (A) at (0,0);
    \coordinate (B) at (2,0);
    \coordinate (C) at (2,2);
    \coordinate (D) at (0,2);

    \draw (A) -- node[above] {$0$} (B)
          -- node[left] {$2$} (C)
          -- node[above] {$3$} (D)
          -- node[left] {$4$} (A);

    \draw (A) -- node[left] {$1$} (C);

    \fill (A) circle (1.5pt)
          (B) circle (1.5pt)
          (C) circle (1.5pt)
          (D) circle (1.5pt);
\end{tikzpicture}
\end{center}
By \Cref{rem: simplifications}, since $M$ is connected, we already know that $m_E = 1$.
Furthermore, the connected flats of $M$ are the singletons, the whole ground set, and the rank 2 flats of size 3 ($\{0,1,2\}$ and $\{1,3,4\}$). We compute the multiplicities of each of these flats. For concreteness, we let $k=2$.
\[
\begin{array}{cccccc}
\toprule  
\text{connected flat } F
&
\rank(F) & \rank(M/F) & \mu^+(M/F) & m_F & 2^{\rank F-1}\beta(M|_F)
\\
\midrule
\{0\},\{2\},\{3\},\{4\} & 1 & 2 & 2 & 8 & 1 \\
\{1\} & 1 & 2 & 1 & 4 & 1 \\
\{0,1,2\}, \{1,3,4\} & 2 & 1 & 1 & 2 & 2\\
\{0,1,2,3,4\} & 3 & 0 & 1 & 1 & 4\\
\bottomrule
\end{array}   
\]
Therefore $E_{L,-2}$ factors as:
\[
E_{L,-2}(z) = \Delta(L^{-1}) \,\Delta(L_{\{0,1,2\}}^{-1})^2 \,\Delta(L_{\{1,3,4\}}^{-1})^2  \, z_0^8 \, z_1^4 \, z_2^8 \, z_3^8 \, z_4^8 \in \KK[z].
\]
This matches degree-wise: $E_{L,-2}$ has degree
\[
(d+1)2^d\mu^+(M)=48 = 4 + 2\cdot 2 + 2\cdot 2 + 8 + 4 + 8 + 8  + 8. \qedhere
\]
\end{example}

\begin{example}[Uniform linear spaces] Let $L \subset \PP^n$ be a generic linear space of dimension $d$. The matroid associated to $L$ is the uniform matroid of rank $d+1$ on $n+1$ elements $U_{d+1,n+1}$.
The only connected flats of $U_{d+1,n+1}$ are the entire ground set $E$ and the singletons. Therefore,
\[
E_{L,-k} = \Delta(L^{-k+1})^{m_E}\prod_{i \in [n]} \Delta(L_{\{i\}}^{-k+1})^{m_i}.
\]
Now, $L_{\{i\}}^{-k+1} = L_{\{i\}}$ is the $i$-th coordinate point. Its dual variety is the set of hyperplanes containing that point, so $\Delta(L_{\{i\}})=z_i \in \KK[z]$.
By \Cref{thm:intro_main}, 
\[
m_i = k^{\rank(M/\{i\})-c(M)+1} \mu^+(M/\{i\}) = k^{d} \mu^+(U_{d+1,n}) = k^{d}\binom{n-1}{d-1}.
\]
As above by \Cref{rem: simplifications} we have $m_E=1$. We obtain in this way the factorization of $E_{L,-k}$:
\[
E_{L,-k} = (z_0z_1\dotsm z_{n-1}z_n)^{k^{d}{\binom{n-1}{d-1}}}\Delta(L^{-k}).
\]
The reader may check that using the following formulas (see \cite[Example 4.2]{briand2025}, and \Cref{th: ppalmatdet}),
\[
\deg(\Delta(L^{-k})) = k^d{\binom{n-1}{d}} \text{ and } \deg(E_{L,-k}) = (d+1)k^{d-c(M)+1}\binom{n}{d},
\]
the following equality holds:
\[
\deg(E_{L,-k}) = \deg(\Delta(L^{-k})) + (n+1)k^{d}{\binom{n-1}{d}}. \qedhere
\]
\end{example}

For the rest of this section we will fix $L \subset \PP^n$ a linear space of dimension $d$ not contained in any coordinate hyperplane, and $k \in \ZZ_{\geq 2}$.

\subsection{A lower bound}\label{sec:lower bound}
\label{sec:lower_bound}
In this section we show that the local multiplicities of $L^{-k}$ along its flat strata are lower bounds on the exponents $m_F$ appearing in the factorization of $E_{L,-k}$. To do so, we pass through the Chow incidence correspondence. Since the methods in this section work in greater generality than just for powers of reciprocal linear spaces, we fix an irreducible, equidimensional projective variety $X\subseteq \PP^n$ of dimension $d$ and an irreducible closed subset $Y\subseteq X$.

We start by fixing some notation.
Let $\mathcal{I}$ be the Chow incidence correspondence
\[
\mathcal{I} \coloneq \Set{(x,[\Lambda]) | x \in X \cap \Lambda } \subseteq X \times \mathbb{G}(n-d-1,\PP^n).
\]
The projection $\pi_1 \colon \mathcal{I} \twoheadrightarrow X$ is surjective while the image of $\pi_2 \colon  \mathcal{I} \rightarrow \mathbb{G}(n-d-1,\PP^n)$ is the Chow variety $\mathcal{C}(X)$. Let $\mathcal{C}_X(Y)$ be the image of $\mathcal{I} \cap (Y \times \mathbb{G}(n-d-1,\PP^n))$ under $\pi_2$, set-theoretically
\[
\mathcal{C}_X(Y) = \Set{ [\Lambda] | Y \cap \Lambda \neq \emptyset } \subseteq \mathbb{G}(n-d-1,\PP^n).
\]
Note that $\mathcal{C}_X(Y) \not\cong \mathcal{C}(Y)$ (they don't have the same dimension), though the latter is a Grassmannian bundle over the former.

\begin{lemma}\label{lem: chowMult}
With the above notations, $\mult_Y(X) \leq \mult_{\mathcal{C}_X(Y)}\mathcal{C}(X)$.
\end{lemma}
\begin{proof}
Let $\Lambda_p \in \mathcal{C}_X(Y)$ be general in the sense of \Cref{prop: propertiesmult}\ref{prop: propertiesmult2}, i.e. 
\[
\mult_{\Lambda_p}\mathcal{C}(X) = \mult_{\mathcal{C}(Y)}\mathcal{C}(X).
\]
Let $p \in \Lambda_p \cap Y$ be any point. We claim that
\[\mult_Y(X)\overset{\text{(a)}}{\leq} 
\mult_p(X) \overset{\text{(b)}}{=} \mult_{(p,\Lambda_p)}(\mathcal{I}) \overset{\text{(c)}}{\leq} \mult_{\Lambda_p}(\mathcal{C}(X)).
\]
\begin{enumerate}[wide,label=(\alph*)]
    \item This inequality follows directly from upper semi-continuity, \Cref{prop: propertiesmult}\ref{prop: propertiesmult2}.
    \item  Over any point $x\in X$, the fiber $\pi_1^{-1}(\{x\})$ consists of all $n-d-1$-dimensional subspaces of $\PP^n$ that contain $x$. Writing $\PP^n=\PP(\KK^{n+1})$ and identifying $x$ with a line in $\KK^{n+1}$, the fiber is hence $\PGr(n-d-2,\PP(\KK^{n+1}/x))$. In other words the first projection $\pi_1$ turns   $\mathcal{I}$ into the relative Grassmann bundle $\PGr(n-d-2,\mathcal{Q}_X)$ over $X$, where $\mathcal{Q}_X$ is the (projectivized) universal quotient bundle of $\PP^n$ restricted to $X$. In particular $\pi_1$ trivializes locally on $X$, i.e. we can find an open affine neighborhood $U \subset X$ of $p$ such that
\[
\pi_1^{-1}(U) \cong U \times \PGr(n-d-2,\PP^{n-1})
\]
Let $(p,\Lambda')$ be the image of $(p,\Lambda_p)$ under the above identification.
By \Cref{prop: propertiesmult}\ref{prop: propertiesmult3},
\begin{align*}
\mult_{(p,\Lambda_p)} \mathcal{I}
&=
\mult_{(p,\Lambda')}(U\times\PGr(n-d-2,\PP^{n-1}))\\ &= (\mult_{p} U) (\mult_{\Lambda'} \PGr(n-d-2,\PP^{n-1})) \\
&= \mult_{p} U\\
&= \mult_p X
\end{align*}
where we use that $\mult_{\Lambda'} \PGr(n-d-2,\PP^{n-1})=1$ since the Grassmannian is smooth, see \Cref{prop: propertiesmult}\ref{prop: propertiesmult1}.
\item The dominant projection $\pi_2: \mathcal{I} \twoheadrightarrow \mathcal{C}(X)$ induces the inclusion of local rings
  \[
  R \coloneq \mathcal{O}_{ \mathcal{C}(X),\Lambda_p} \hookrightarrow \mathcal{O}_{\mathcal{I},(p,\Lambda_p)} \eqcolon S
  \]
  Let $\mm_R, \mm_S$ be the maximal ideals of $R,S$ respectively. We define $J \coloneq \mm_RS$ as the ideal generated by $\mm_R$ in $S$ under the above inclusion. We have, by definition of multiplicity at a point, 
  $\mult_{(p,\Lambda_p)}(\mathcal{I}) = e(\mm_S,S)$ and $\mult_{\Lambda_p}(\mathcal{C}(X)) = e(\mm_R,R)$. Since $X$ is irreducible, so are $\mathcal{C}(X)$ and $\mathcal{I}$ and hence $R$, $S$ are integral domains. Furthermore $S$ is a rank 1 $R$-module since the projection map $\pi_2$ is generically one-to-one, so we can apply \Cref{thm: ZariskiSamuel}:
\begin{align*}
   e(\mm_R,R)= \underbrace{[S:R]}_{=1}e(\mm_R,R)=\underbrace{[S/\mm_S:R/\mm_R]}_{=[\KK:\KK]=1}e(J,S)=e(J,S)\geq e(\mm_S,S)
\end{align*}
Here the last inequality follows from \Cref{prop: locmult} \ref{prop: locmult1} using $J\subseteq \mm_S$. This concludes the proof since by definition $e(\mm_R,R)=\mult_{\Lambda_p}(\mathcal{C}(X))$ and $e(\mm_S,S)=\mult_{(p,\Lambda_p)}(\mathcal{I})$. \qedhere
\end{enumerate}
\end{proof}

\begin{lemma}\label{lem: multdommap}
    Let $X,X'$ be smooth varieties. Let
    \[
    \varphi: X' \dashrightarrow X
    \]
    be a rational map and suppose that $Y \subset X$ and its (scheme-theoretic) preimage $\varphi^{-1}(Y)\subset X'$ have codimension 1.
    Pick $p\in \varphi^{-1}(Y)$ such that $\varphi$ is defined at $p$. Then
    \[
    \mult_p(\varphi^{-1}(Y)) \geq \mult_{\varphi(p)}(Y)
    \]
\end{lemma}
\begin{proof}

   Since multiplicity is a local property, we can assume that $\varphi$ is a regular map. The map $\varphi$ induces a morphism of local rings $\varphi^*:\mathcal{O}_{X,\varphi(p)}\to \mathcal{O}_{X',p}$.
    We use the characterization of multiplicity for hypersurfaces given in \Cref{prop: propertiesmult} \ref{prop: propertiesmult5}.

    Let $f \in \mathcal{O}_{X,\varphi(p)}$ be the regular function defining $Y$ in a neighbourhood of $\varphi(p)$. Then $\varphi^*f \in \mathcal{O}_{X',p}$ is a regular function defining $\varphi^{-1}(Y)$ in a neighbourhood of $p$. Since $X$ and $X'$ are smooth, $\mathcal{O}_{X,\varphi(p)}$ and $\mathcal{O}_{X',p}$ are regular so we are in the situation of \Cref{prop: propertiesmult} \ref{prop: propertiesmult5}.

    Now suppose that $l\geq 0$ is such that
    \[
    f \in \mm^l_{X,\varphi(p)}
    \]
  %  Then we can write
  %  \[
  %  f = \sum_{i \in I} h_ig_i^{(1)}\cdots g_i^{(k)}
  %  \]
   % where $I$ is a finite set and for every $i \in I$, $h_i \in \mathcal{O}_{X,\varphi(p)}$ and for every $j = 1,\cdots,k$, $g_i^{(j)} \in \mm_{X,\varphi(p)}$.
   % Since $\varphi^{-1}(\mm_{X,\varphi(p)})$ is a proper ideal of $\mathcal{O}_{X',p}$, we have that $\varphi^{-1}(\mm_{X,\varphi(p)}) \subset \mm_{X',p}$ and therefore
   % $
   % \varphi^*g_i^{(j)} \in \mm_{X',p}, \quad \forall i \in I, j = 1,\cdots,k
   % $
  % This implies that
  % \[
  % \varphi^*f = \varphi^*(\sum_{i \in I} h_ig_i^{(1)}\cdots g_i^{(k)}) = \sum_{i \in I} (\varphi^*h_i)(\varphi^*g_i^{(1)})\cdots (\varphi^*g_i^{(k)}) \in \mm_{X',p}^k
   %\]
   Then since $\varphi^*$ is a morphism of rings,
   \[
   \varphi^*f \in \varphi^{-1}(\mm_{X,\varphi(p)})^l=\mm_{X',p}^l
   \]
   Since this holds for every $k \geq 0$ we conclude that $\ord_{\mm_{X,\varphi(p)}}(f) \leq \ord_{\mm_{X',p}}(\varphi^*f)$ and thus
   \[
    \mult_p(\varphi^{-1}(Y)) \geq \mult_{\varphi(p)}(Y) \qedhere
   \]
\end{proof}

We now apply the above lemmas to our specific situation. For a flat $F$ of $M = M(L)$, recall that $T_F = \{[p_0:\cdots:p_n] \in \PP^n \text{ such that } p_i = 0 \text{ iff } i \notin F\} \subset \PP^n$ and $\Lambda_F = \overline{T_F}$.

\begin{corollary}\label{cor: lowerbound1}
Let $k\geq 2$. For every flat $F$ of $M$ we have the inequality
\[
   \mult_{L^{-k}_F}(L^{-k}) \leq m_F
\]
where $m_F$ denotes the multiplicity of the factor $\Delta(L_F^{-k+1})$ in the polynomial $E_{L,-k}$.
\end{corollary}
\begin{proof}
Denote by $\mathcal{E}_{L,-k}$ the scheme defined by $E_{L,-k}$, then we have $m_F=\mult_{\nabla(L_F^{-k+1})}\mathcal{E}_{L,-k}$ (see \Cref{ex:multiesgeovsalg}).
 We first apply \Cref{lem: chowMult} to $X = L^{-k}$ and $Y = L^{-k}_F$ and obtain that
    \[
    \mult_{L^{-k}_F}(L^{-k}) \leq \mult_{\mathcal{C}_{L^{-k}}(L^{-k}_F)}\mathcal{C}(L^{-k})
    \]
    Now consider the map 
    \[
    r \circ i: \PP^n \dashrightarrow \PGr(n-d-1,\PP^n), \quad z \mapsto \PP(\ker(A\cdot \text{diag}(z))).
    \]
    and notice that by definition $\mathcal{E}_{L,-k}$ is the scheme-theoretic preimage of $\mathcal{C}(L^{-k})$ under this map.
    Pick $p\in \nabla(L_F^{-k+1})$ generic as in \Cref{prop: propertiesmult}\ref{prop: locmult2} such that $\mult_p(\mathcal{E}_{L,-k})=\mult_{(L_F^{-k+1})^\vee}(\mathcal{E}_{L,-k})$. Notice that, analogously to the proof of \Cref{th: ppalmatdet} we have
    \[
    \overline{(r \circ i)(((L_F)^{-k+1})^\vee)} = \mathcal{C}_{L^{-k}}(L^{-k}_F)
    \]
    and therefore $(r\circ i)(p)\in \mathcal{C}_{L^{-k}}(L^{-k}_F)$. Since both $\mathcal{C}(L^{-k})$ and $\mathcal{E}_{L,-k}$ are hypersurfaces, we can apply \Cref{lem: multdommap} to $X=\PGr(n-d-1,\PP^n), X'=\PP^n, \varphi=r\circ i$ to obtain 
    \[
    \mult_{p}(\mathcal{E}_{L,-k}) \geq \mult_{r \circ i(p)}(\mathcal{C}(L^{-k})).
    \]
    By upper semi-continuity \Cref{prop: propertiesmult}\ref{prop: propertiesmult2} we have 
    \[
    \mult_{r \circ i(p)}(\mathcal{C}(L^{-k}))\geq \mult_{\mathcal{C}_{L^{-k}}(L^{-k}_F)}\mathcal{C}(L^{-k})
    \]
    which finishes the proof.
\end{proof}

\subsection{Local multiplicities of \texorpdfstring{$L^{-k}$}{L\^{}-1}}\label{sec:local mult of L-k}
\label{sec:computing_multi}
Next, we compute the multiplicity $\mult_{L^{-k}_F}(L^{-k})$. An important ingredient is the degree of the coordinatewise $k$-th power map after restricting to $L^{-1}$, which we briefly recall here. Remember that $c(M)$ denotes the number of connected components of $M$.
\begin{lemma}\label{lem:degreeofkthpower}
Consider the coordinatewise $k$-th power map \begin{align*}
    h_k:L^{-1}\to L^{-k}, \qquad x\mapsto x^k.
\end{align*}
This map is finite to one with $\deg(h_k)=k^{c(M)-1}$. In particular the corresponding map of affine cones $\widehat{L^{-1}}\to \widehat{L^{-k}}$ is finite to one with degree $k^{c(M)}$
\end{lemma}
\begin{proof}
    Since the rational map which inverts coordinates is one-to-one, the degree of $h_k$ equals the degree of $\widetilde{h}_k:L\to L^k,x\mapsto x^k$. This degree can be easily computed using \cite[Theorem 2.6]{Dey2020}. For more details we refer to the proof of \cite[Lemma 3.2]{briand2025}, where the argument is carried out for the case $k=2$. The additional factor of $k$ in the affine case comes from the fact that each affine representative of a point in a fixed fiber can be rescaled by the $k$ different $k$-th roots of unity which does not change the projective point.
\end{proof}

\begin{proposition}\label{lem: lowerbound2}
For every flat $F$ of $M$ we have the equality    
\[
\mult_{L^{-k}_F}(L^{-k})=k^{\rank(M/F)-c(M)+c(M|_F)}\mu^+(M/F).
\]
\end{proposition}

\begin{proof}
    Without loss of generality we may assume $F=\{0,\ldots,s\}$ for some $s<n$. By \Cref{prop: propertiesmult}\ref{prop: propertiesmult4} we have $\mult_{L^{-k}_F}(L^{-k})=\mult_{\widehat{L^{-k}_F}}(\widehat{L^{-k}})$ and hence we pass to affine cones.
    Fix $q\in \widehat{L^{-k}_F}$ such that both \begin{itemize}
        \item $\mult_{\widehat{L^{-k}_F}}(\widehat{L^{-k}})=\mult_q(\widehat{L^{-k}})$, see \Cref{prop: propertiesmult}\ref{prop: propertiesmult2}, and
        \item in coordinates we have $q=(q_0,\ldots,q_s,0,\ldots,0)$ for $q_0,\ldots,q_s\in \KK^\times$, i.e. $q\in \widehat{T_F}$.
    \end{itemize}
     The surjective $k$-th power map $\widehat{L^{-1}}\to \widehat{L^{-k}}$ induces the following inclusion of affine coordinate rings: \begin{align*}
        R'\coloneq \KK[y_0,\dots,y_n]/I(\widehat{L^{-k}})&\hookrightarrow\KK[x_0,\dots, x_n]/I(\widehat{L^{-1}})\eqcolon S'  \\
        y_i&\mapsto x_i^k
    \end{align*}
    Since the $k$-th power map is finite, $S'$ is a finitely generated $R'$-module under this map.
    
    The maximal ideal corresponding to $q$ is $\widetilde{\mm_q}=\langle y_1-q_1,\ldots y_s-q_s,y_{s+1},\ldots,y_n \rangle$.
    Localizing at this maximal ideals we obtain the inclusion \begin{align*}   R\coloneq\left(\KK[y_1,\ldots,y_n]/I(\widehat{L^{-k}})\right)_{\widetilde{\mm_q}}&\hookrightarrow\left(\KK[x_1,\cdots x_n]/I(\widehat{L^{-1}})\right)\otimes_{R'}R\eqqcolon S
    \end{align*}
    We set $\mm_q\coloneq\widetilde{\mm_q}R$. Under the inclusion of $R$ into $S$ we define $J\coloneq m_qS$. $S$ is a semi-local ring with maximal ideals $\mm_p=\langle x_0-p_0,\ldots, x_s-p_s,x_{s+1},\ldots,x_n \rangle S$ for any point $p\in \widehat{L^{-1}}$ with $p^k=q$. Write $J_p$ for the $\mm_p$-primary component of $J$, then we have  \begin{align*}
        J_pS_{\mm_p}&=\langle x_0^k-q_0,\ldots,x_s^k-q_s,x_{s+1}^k,\ldots,x_n^k\rangle S_{\mm_p} \\
        &=\langle (x_0-p_0)f_0,\ldots,(x_s-p_s)f_s,x_{s+1}^k,\ldots,x_n^k\rangle S_{\mm_p} \\
        &=\langle x_0-p_0,\ldots,x_s-p_s,x_{s+1}^k,\ldots,x_n^k\rangle S_{\mm_p} 
    \end{align*}
    where $f_i=\sum_{j=0}^kx_i^jp_i^{k-j}$.
    In the last step we used that the polynomials $f_i$ are in the complement of the maximal ideal $\mm_p$ and hence are units in the local ring $S_{\mm_p}$.

    By choice of $q$ we have $\mult_{L^{-k}_F}(L^{-k})=\mult_q(\widehat{L^{-k}})=e(\mm_q,R)$ and by
    \Cref{thm: ZariskiSamuel} we get \begin{align*}
        [S:R]e(\mm_q,R)=\sum_{\substack{p\in L^{-1} \\ p^k=q}}[S/\mm_p:R/\mm_q]e(J_p,S)
    \end{align*}
    Here $[S:R]$ equals the degree of the finite $k$-th power map on affine cones, which is $k^{c(M)}$ due to \Cref{lem:degreeofkthpower}. On the other hand we have $[S/\mm_p:R/\mm_q]=[\KK:\KK]=1$. Therefore it suffices to show $e(J_p,S)=k^{\rank(M/F)}\mu^+(M/F)$ for every $p\in \widehat{L^{-1}}$ with $p^k=q$. We defer this lengthy computation to \Cref{lem: splittingmultiplicity}. Indeed applying again \Cref{lem:degreeofkthpower}, there are precisely $k^{c(M|_F)}$ preimages $p\in \widehat{L^{-1}}$ of the generic point $q\in L\cap T_F$ under the squaring map. Hence given the result of \Cref{lem: splittingmultiplicity} we get \begin{align*}
        k^{c(M)}\mult_{L^{-k}_F}(L^{-k})=\sum_{\substack{p\in L^{-1} \\ p^k=q}}k^{\rank(M/F)}\mu^+(M/F)=k^{c(M|F)+\rank(M/F)}\mu^+(M/F)
    \end{align*}
    which finishes the proof after dividing by $k^{c(M)}$ on both sides.
\end{proof}

\begin{lemma}\label{lem: splittingmultiplicity}
    Using the notations of the proof of \Cref{lem: lowerbound2}, fix $p \in L^{-1}$ such that $p^k=q$. We have that
    \[
    e(J_p,S)=k^{\rank(M/F)}\mu^+(M/F)
    \]
\end{lemma}

\begin{proof}
By definition $e(J_p,S)=e(J_pS_{\mm_p},S_{\mm_p})$.
By \cite[Theorem 3.3]{ELIAS201636} there exists an open set $W_F^\circ \subset  \widehat{L^{-1}}$ containing $p$ and a map
\begin{align*}
\phi: W_F^\circ \rightarrow \widehat{(L_F)^{-1}} \times \widehat{(L/F)^{-1}}
\end{align*}
which sends $p$ to $((p_0,\cdots,p_s),0)$ and is étale at $p$. Here $L/F$ denotes the intersection of $L$ with the coordinate hyperplanes corresponding to $F$: $L/F=L\cap\bigcap_{i\in F}\{x_i=0\}$.

It induces a morphism of local rings
$\phi^{\#}\colon M_{\widetilde{\mm}_p} \rightarrow S_{\mm_p}$
where
\begin{align*}
M \coloneq 
\underbrace{(\KK[z_0,\cdots,z_s]/I(\widehat{(L_F)^{-1}}))_{\langle z_0-p_0,\ldots,z_s-p_s\rangle}}_{\eqcolon M_1} \otimes_\KK \underbrace{(\KK[z_{s+1},\cdots,z_n]/I(\widehat{(L/F)^{-1}}))_{\langle z_{s+1},\ldots,z_n \rangle} }_{\eqcolon M_2} 
%_{\langle x_1-p_1,\cdots,x_s-p_s,x_{s+1},\cdots,x_n\rangle}
%=((\KK[z_0,\cdots,z_s]/I(L|F))_p \otimes ( \KK[z_{s+1},\cdots,z_n]/I(L/F))_p
\end{align*}
and
\[
\widetilde{\mm}_p =
\langle z_0-p_0,\cdots, z_s-p_s,z_{s+1},\cdots, z_n \rangle M.
\]

Let $J_p'\subseteq M$ be the ideal
\begin{align*}
J_p' &\coloneq \langle z_0-p_0, \cdots, z_s-p_s,z_{s+1}^k, \cdots z_n^k\rangle M 
\end{align*}
The ideal generated by $\phi^\#(J_p'M_{\widetilde{\mm}_p})$ in  $S_{\mm_p}$ is
\begin{align*}
 \phi^\#(J_p'M_{\widetilde{\mm}_p})S_{\mm_p} &= \langle \phi^\#(z_0-p_0), \cdots, \phi^\#(z_s-p_s),\phi^\#(z_{s+1}^k), \cdots, \phi^\#(z_n^k)\rangle S_{\mm_p}\\&=  \langle \phi^\#(z_0)-p_0, \cdots, \phi^\#(z_s)-p_s, \phi^\#(z_{s+1})^k, \cdots \phi^\#(z_n)^k\rangle S_{\mm_p}\\&=J_p
\end{align*}
The last equality is because by construction of $\phi$ (\cite[Theorem 3.3]{ELIAS201636}), for $0 \leq i \leq s$,
\[
\phi^\#(z_i) = x_i
\]
and for $s+1 \leq i \leq n$
\[
\phi^\#(z_i) = \frac{x_i}{1-b_ix_i}
\]
where $b_i\in \KK[\widehat{L^{-1}\cap T_F}]$. Notice that $1-b_ix_i \in S_{\mm_p}\setminus \mm_p$, so it is a unit in $S_{\mm_p}$. 

Since the map $\phi^\#$ is étale, by \Cref{prop: locmult}\ref{prop: locmult2},
\begin{align*}
e(J_p,S_{\mm_p}) =e(\phi^\#(J_p'M_{\widetilde{\mm}_p})S_{\mm_p},S_{\mm_p})= e(J_p'M_{\widetilde{\mm}_p},M_{\widetilde{\mm}_p})=e(J_p',M)
\end{align*}
%Multiplicity is a local statement, therefore it does not change when localizing and so
%\begin{align*}
%e(J,S) = e(J',M_1 \otimes_\KK M_2)
%\end{align*}
Now since $J_p'$ splits as $J_p'=J_1\otimes_{\KK}M_2+M_1\otimes_{\KK}J_2$ where
\begin{align*}
J_1 &= \langle z_0-p_0,\cdots,z_s-p_s\rangle\subseteq M_1,\\
J_2 &= \langle z_{s+1}^k, \cdots, z_n^k \rangle \subseteq M_2,
\end{align*}
by \Cref{prop: locmult}\ref{prop: locmult3}, we get
\begin{equation}\label{eq: J'prod}
e(J_p',M) = e(J_1,M_1)e(J_2,M_2).
\end{equation}

By \Cref{ex: multinRLS}, since $(p_0,\cdots,p_s) \in \widehat{T_F} \cap \widehat{(L_F)^{-1}}$,
\begin{align*}
    e(J_1,M_1) = \mu^+((M|_F)/F) = \mu^+(\emptyset) = 1.
\end{align*}
Alternatively this follows from \Cref{prop: propertiesmult}\ref{prop: propertiesmult1} since the point $(p_0,\ldots,p_s)$ has non-zero coordinates and is therefore a smooth point of $(L_F)^{-1}$.
It remains to determine $e(J_2,M_2)$.

Let $J_3 \coloneq \langle z_{s+1},\cdots, z_n \rangle\subseteq M_2$. Notice that for any $l \geq 0$, $((J_3)^k)^{l+1} $ is generated by all monomials of degree $k(l+1)$, while $ (J_3^k)^lJ_2$  is generated by all monomials of degree $k(l+1)$ such that at least one variable appears with exponent greater or equal to $k$. If $k(l+1) >(k-1)(n-s)$, these two generating sets coincide. Since additionally $J_3^k \supset J_2$, $J_2$ is a reduction of $J_3^k$ in the sense of \cite{Rees_1961}. It follows that
$$
e(J_2,M_2) = e(J_3^k,M_2).
$$
We wish to relate $e(J_3^k,M_2)$ to $e(J_3,M_2)$ which is simply the degree of $(L/F)^{-1}$. To do so note that for any $i=0,\dots, k-2$ and any $l>0$ we have
\[
\left(J_3^{kl+i}/J_3^{kl+k}\right)/\left(J_3^{kl+i+1}/J_3^{kl+k}\right)=J_3^{kl+i}/J_3^{kl+i+1}
\]
as $M_2$-modules and hence in particular as $\KK$-vector spaces. We therefore get 
\[
\dim\left(J_3^{kl+i}/J_3^{kl+k}\right)=\dim(J_3^{kl+i}/J_3^{kl+i+1})+\dim\left(J_3^{kl+i+1}/J_3^{kl+k}\right).
\]
Combining these equations for all $i=0,\dots,k-2$ we deduce 
\[
\dim((J_3^k)^l/(J_3^{k})^{l+1})=\dim(J_3^{kl}/J_3^{kl+k})=\sum_{i=0}^{k-1}\dim(J_3^{kl+i}/J_3^{kl+i+1}).
\]
For $l\gg 0$ this implies the following equality of Hilbert polynomials:
\[
\text{HP}_{J_3^k,M_2}(t)=\sum_{i=0}^{k-1}\text{HP}_{J_3,M_2}(kt+i)\in \ZZ[t].
\]
All Hilbert polynomials appearing in this equation have degree $\dim(M_2)-1$, hence by comparing leading terms and multiplying by $(\dim(M_2)-1)!$ we obtain \begin{align*}
    e(J_3^k,M_2) &= \sum_{i=0}^{k-1}k^{\dim(M_2) -1}e(J_3,M_2) \\&= k^{\dim(M_2)}e(J_3,M_2)\\ &= k^{\rank(M/F)}e(J_3,M_2)
\end{align*}

By definition $e(J_3,M_2)$ is the local multiplicity of the affine cone over $(L/F)^{-1}$ at the origin and hence equals the degree of the projective variety $(L/F)^{-1}$. This degree is $\mu^+(M/F)$ and hence we conclude that 
\[
    e(J_2,M_2) =e(J_3^k,M_2)= k^{\rank(M/F)}\mu^+(M/F). \qedhere
\]
\end{proof}

\subsection{Matching the lower bound}\label{sec:match}

In this section we finish the proof of \Cref{thm:intro_main} by showing that the lower bound we derived in the previous sections is actually an equality. The main ingredient will be the following lemma.

\begin{lemma}\label{lem:combinatorialErecDeg}
Let $M$ be any matroid, then
\[
\rank(M) \cdot \mu^+(M) = \sum_{F} \beta(M|_F)\cdot \mu^+(M/F),
\]
where the summation run over all subsets, equivalently all connected flats of $M$.
\end{lemma}
For the proof we recall the fundamentals of Möbius inversion for matroids \cite[Chapter 7]{WhiteCombGeo1987}. Let $f,g$ be functions that take as input a matroid and return a complex number. Their convolution product is defined as
\[
(f*g)(M) \coloneq \sum_{F \text{ flat of }M} f(M|_F)g(M/F).
\]
This product is associative, \emph{not} commutative, and satisfies \emph{Möbius inversion}:
\[
(\mu * 1)(M) = \delta_{\emptyset,M}, \qquad (1 * \mu)(M) = \delta_{0,\rank(M)}.
\]

\begin{proof}[Proof of \Cref{lem:combinatorialErecDeg}]
Since $\beta(M|_F)\mu^+(M/F)=0$ unless $F$ is a connected flat, we may restrict our attention to the summation over all flats of $M$. The equation in \Cref{lem:combinatorialErecDeg} then asserts that ${\rank} \cdot \mu^+ = \beta * \mu^+$.
Define the \enquote{signed $\beta$-invariant} as
\[
\beta^\pm(M) \coloneq (-1)^{\rank M} \beta(M) = -\chi'_M(1) = \sum_{F \text{ flat of }M} \mu(M|_F) \rank(M|_F).
\]
Replacing $\mu,\beta^\pm$ by $\mu^+,\beta$, our desired identity then reads ${\rank} \cdot \mu = \beta^{\pm} * \mu$.
The last sum in the defining equation of $\beta^\pm$ shows that $\beta^\pm = ({\rank} \cdot \mu) * 1$. Applying ${}*\mu$ and using Möbius inversion yields the desired formula
\[
\beta^\pm * \mu = ({\rank} \cdot \mu) * 1 * \mu = ({\rank} \cdot \mu) * \delta_{0,\rank(-)} = {\rank} \cdot \mu. \qedhere
\]
\end{proof}

\begin{proof}[Proof of \Cref{thm:intro_main}]
By definition of the $m_F$ in \Cref{th: ppalmatdet},
\[
\deg(E_{L,-k}) = \sum_{F \text{ conn.\ flat of }M} m_F \cdot \deg(L_F^{-k+1})^{\vee}.
\]
Combining \Cref{cor: lowerbound1} and \Cref{lem: lowerbound2}, for any flat $F$ of $M$,
\[
m_F \geq \mult_{L_F^{-k}}(L^{-k}) = k^{\rank(M/F)-c(M)+1}\mu^+(M/F).
\]
We also know from \cite[Theorem 4.8]{briand2025} that $\deg\,(L_F^{-k+1})^{\vee} = k^{\rank(M|_F)-1}\beta(M|_F)$ whenever $F$ is a connected flat, hence
\begin{align}\label{eq: ineqdegs}
\begin{split}
m_F\cdot \deg(L_F^{-k+1})^{\vee} &\geq  k^{\rank(M/F)-c(M)+1}\mu^+(M/F) \cdot k^{\rank(M|_F)-1}\beta(M|_F) \\
&= k^{\rank(M)-1-c(M)+1}\mu^+(M/F) \beta(M|_F).
\end{split}
\end{align}
Therefore, using our notation $d = \rank M-1$,
\begin{align*}
\deg(E_{L,-k}) &\geq \sum_{F \text{ conn.\ flat of }M} k^{d-c(M)+1}\mu^+(M/F) \beta(M|_F) \\
    &= k^{d-c(M)+1} \sum_{F \text{ flat of }M} \mu^+(M/F) \beta(M|_F)\\
    &= k^{d-c(M)+1} \rank(M) \mu^+(M)
\end{align*}
where the last equality uses \Cref{lem:combinatorialErecDeg}.
But we already know from \Cref{th: ppalmatdet} that $\deg(E_{L,-k}) = k^{d-c(M)+1}\rank(M) \mu^+(M)$. Therefore, the above inequality is in fact an equality, and thus the inequalities in equations \eqref{eq: ineqdegs} for \emph{all flats} $F$ must be equalities too.  In particular,
\[
m_F = \mult_{L^{-k}_F}(L^{-k})=k^{\rank(M/F)-c(M)+1}\mu^+(M/F)
\]
for all connected flats $F$ of $M$.
% This proves in particular \Cref{con:multies}. Indeed, by \Cref{lem: lowerbound2}, for any connected flat $F$ of $M$, the local multiplicity of $L^{-2}$ along the stratum indexed by $F$ is $2^{\rank(M/F)-c(M)+1}\mu^+(M/F)$, which by the previous theorem is the exponent of the defining polynomial of $((L|F)^{-1})^\vee$ in the factorization of $E_L$.
\end{proof}

\section{Other powers of linear spaces}\label{sec:otherMulties}

We briefly give some results on the generalised matroid determinant $E_{L,-k}$ when $k<2$. We start with the case $k=1$ where $E_{L,-1}=i^*r^*\text{Ch}(L^{-1})$ for the maps $i,r$ from \Cref{def:matroid_det}.
\begin{proposition}\label{lem: k=-1}
The polynomial $E_{L,-1}$ has degree $\deg(E_{L,-1}) = (d+1)\mu^+(M)$ and factorizes as
\[
E_{L,-1}(z) = \prod_{\substack{F \in \mathcal{F}(M) \\ F \text{ connected}}} (\sum_{i \in F}z_i)^{m_{F,1}} \in \KK[z].
\]
for some integers $m_{F,-1}\geq 1$.
\end{proposition}
\begin{proof}
The degree formula for $E_{L,-1}$ follows from a similar argument as \cite[Theorem 3.5]{MatsubaraHeoTelen2026}:
\[
\deg(E_{L,-1}) = \deg(i)\deg(r)\deg(\Chow(L^{-1}))=(d+1)\deg(L^{-1})= (d+1)\mu^+(M).
\]

    For the second claim, fix a flat $F \in \mathcal{F}(M)$, and define
    \[
    Y_F \coloneq \set{(x,z) \in \PP^n \times (\PP^n)^* \mid A_z\cdot x = 0 \text { and } x \in L^{-1} \cap T_F}.
    \]
    Since $x_i = 0$ for $i \notin F$, the condition $A_z\cdot x = 0$ becomes
    \begin{align*}\label{eq: scalarprod}
    \sum_{j \in F} a_{ij}x_jz_j = 0 \quad i = 0,\dots,d,
    \end{align*}
    or equivalently the vector $(x_iz_i)_{i\in F}$ belongs to $L_F^\perp=\{y\in \PP(\KK^F)\mid \forall w\in L_F:\sum_{i\in F}y_iw_i=0\}$. Notice that for a fixed $z\in (\PP^n)^*$ with the property that not all $z_i=0$ for $i\in F$, there exists $x\in L^{-1}\cap T_F$ with the above property if and only if $(z_i)_{i\in F}\in (L_F\cap T_F)\star L_F^\perp$. This shows that $\overline{\pi_2(Y_F)}=(L_F\star L_F^\perp)\vee \PP(\KK^{F^c})$ where $\vee$ denotes the join of two projective varieties.

  %  This holds if and only if $ \sum_{j \in F} x_jz_j \in L_F^\perp$, or equivalently, since $x \in (L_F^\circ)^{-1}$, if $z \in L_F^\circ \star L_F^\perp \times \KK^{F^c}$, where $\star$ denotes the Hadarmard product.

    Now, the stratification of $L^{-1}$ by flats induces the stratification
    \[
    Y \coloneq \bigsqcup_{F \in \mathcal{F}(M)} Y_F \subset \PP^n \times (\PP^n)^*
    \]
    whose image under the projection $\pi_2:\PP^n \times (\PP^n)^* \rightarrow (\PP^n)^*$ is the vanishing locus of $E_{L,-1}$ by the same reasoning as in \cite[Lemma 3.7]{MatsubaraHeoTelen2026}. The above argument shows that
    \[
    \VV(E_{L,-1}) = \pi_2\big(\bigsqcup_{F \in \mathcal{F}(M)} Y_F\big) = \bigcup_{F \in \mathcal{F}(M)}(\overline{\pi_2(Y_F)}) = \bigcup_{F \in \mathcal{F}(M)} (L_F \star L_F^\perp)\vee \PP(\KK^{F^c}).
    \]
    Now, by \cite[Proposition 4.6]{MatsubaraHeoTelen2026} and \cite[Theorem 2.11]{briand2025}, $\dim(L_F \times L_F^\perp) = |F|-2$ if and only if $F$ is a connected flat, and in that case, it is the hyperplane given by $\sum_{i \in F} z_i = 0$. This completes the argument.
\end{proof}

\begin{remark}
    Define $L^0 \coloneq \{\bm{1}\}$ where $\bm{1}$ is the point with all coordinates equal to one. Then, 
    \[
    E_{L,-1} = \prod_{F \text{ connected flat}} \Delta(L_F^0)^{m_F}
    \]
    for some integers $m_F \geq 0$, matching the general expression for the factorization of the generalized matroid determinant in \Cref{th: ppalmatdet}.
\end{remark}

The proof of our main theorem does not extend to the case $k=1$. In fact all arguments in sections \ref{sec:lower_bound} and \ref{sec:computing_multi} can be adapted also to this case, but the obtained inequality is in general strict. However, there exists a determinantal formula for the Chow form of $L^{-1}$ due to \cite[Corollary 3.9, Remark 3.3]{kummer}. Using this formula we verified the following conjecture for uniform matroids and computationally for small non-uniform matroids. We also remark that this determinantal formula implies that $E_{L,-1}$ depends only on the (labeled) matroid and not on the realisation $L$.  

\begin{conjecture}
The multiplicities in \Cref{lem: k=-1} satisfy
\[
m_{F,-1}=\mu^+(M/F)\beta(M|_F)
\]
for every connected flat $F$ of $M$.
\end{conjecture}

Finally, when $k\leq -1$ we have the same degree computation as in the previous case with exactly the same proof.

\begin{lemma}\label{lem: positivek}
Let $k \in \ZZ_{\geq 1}$ and fix a linear subspace $L\subseteq \PP^n$ be of dimension $d$. Set $E_{L,k}\coloneq i^*r^*\Chow(L^k)$ as in \Cref{def:matroid_det}.
Then
\[
\deg(E_{L,k}) = (d+1)k^{d-c(M)+1}.
\]
\end{lemma}

Based on computational experiments, it seems that, assuming $M$ is connected, $E_{L,k}$ is an irreducible polynomial equal to the defining equation of the dual hypersurface $(L^{k+1})^\vee$. A similar method as has been used in the proof of \Cref{th: ppalmatdet} can show that the zero locus of $E_{L,k}$ still contains $(L^{k+1})^\vee$. We note also that in this case the analogous results from \cite{briand2025} (most notably the dual degree of $L^{k}$), which we have used throughout our article, have not yet been established, compare \cite[Remark 4.13]{briand2025}. 

\begin{conjecture}
\label{conj: positivek}
    Let $k\geq 1$ and let $L\subseteq \PP^d$ be of dimension $d$ and such that its matroid $M$ is connected, then $E_{L,k}=\Delta(L^{k+1})$. In particular $\deg(L^{k+1})^\vee=(d+1)k^d$.
\end{conjecture}

\begin{example}
For $k=1$, one can deduce \Cref{conj: positivek} from elementary considerations. A direct calculation shows that 
\[
\Res(L)(u) = \det(UA^{\textsf{T}}), \qquad E_{L,1}(z) = \det(A\operatorname{diag}(z) A^{\textsf{T}}) = \sum_{S \in \binom{[n+1]}{d+1}} \det(A_S)^2 z^S,
\]
where the last formula is the Cauchy--Binet formula. This expansion shows that $S$ is in the support of $E_{L,1}$ if and only if $S$ is a basis of the matroid $M=\Matroid(L)$. If $M$ is connected, then such a polynomial must be irreducible, see \cite[Proposition 4.6]{Choe2004}. As $\Delta(L^2) \mid E_{L,1}$, these two polynomials must be equal up to scaling.
\end{example}

\printbibliography
\end{document}